\documentclass[12pt,a4paper]{article}

\usepackage{amssymb, amsmath, amsthm}

\usepackage[cm]{fullpage}
\usepackage[english]{babel}
\usepackage[T1]{fontenc}
\usepackage[pdftex]{graphicx}
\usepackage{booktabs}
\usepackage{xcolor}

\usepackage{hyperref}
\usepackage[normalem]{ulem}

\newtheorem{theorem}{Theorem}
\newtheorem{lemma}{Lemma}
\newtheorem{proposition}{Proposition}
\newtheorem{remark}{Remark}
\newtheorem{corollary}{Corollary} 

\title{Functional equation arising in behavioral sciences: solvability and collocation scheme in H\"older spaces}
\author{
Josefa Caballero\thanks{Departamento de Matem\'aticas, Universidad de Las Palmas de Gran Canaria, Campus de Tafira Baja, $35017$ Las Palmas de Gran Canaria, Spain.}, \and
Hanna Okrasi{\'n}ska-P{\l}ociniczak\thanks{Department of Mathematics, Wroclaw University of Environmental and Life Sciences, ul. C.K. Norwida 25, 50-275 Wroclaw, Poland}, \and 
{\L}ukasz P{\l}ociniczak\thanks{Faculty of Pure and Applied Mathematics, Wroclaw University of Science and Technology, Wyb. Wyspia\'nskiego 27, 50-370 Wroc{\l}aw, Poland, \underline{corresponding author:} \texttt{lukasz.plociniczak@pwr.edu.pl}}, \vspace{4pt}\and
Kishin  Sadarangani$^{\ast}$ 
}
\date{}

\begin{document}
\maketitle

\begin{abstract}	
	We consider a generalization of a functional equation that models the learning process in various animal species. The equation can be considered nonlocal, as it is built with a convex combination of the unknown function evaluated at mixed arguments. This makes the equation contain two terms with vanishing delays. We prove the existence and uniqueness of the solution in the H\"older space which is a natural function space to consider. In the second part of the paper, we devise an efficient numerical collocation method used to find an approximation to the main problem. We prove the convergence of the scheme and, in passing, several properties of the linear interpolation operator acting on the H\"older space. Numerical simulations verify that the order of convergence of the method (measured in the supremum norm) is equal to the order of H\"older continuity. \\
	
	
\end{abstract}

\section{Introduction}
Since the middle of the twentieth century, psychologists have considered mathematics to be a useful tool in the learning process \cite{brunswik1939probability}. The main motivation for our work comes from the results of other authors that were published in \cite{turab2019analytic,turab2020corrigendum}, where a model describing the learning process of paradise fish was analyzed. This, in turn, is based on empirical research described in \cite{bush1956two}. In the experiment, the fish were given two gates through which to swim. One of them had a higher probability that a fish would obtain a reward. In this way, it was natural to observe that in subsequent trials the fish preferred the more rewarding gate over the other. If $t$ is the initial probability of choosing the most rewarding gate, it should increase to $\alpha t + 1-\alpha$ in the next trial when the fish chooses it. On the other hand, this probability decreases to $\beta t$ when the fish chooses the least beneficial gate. Here, $0<\alpha\leq\beta<1$ are the learning rates. A specimen adapting the strategy of choosing the most rewarding gate should then gain an advantage over other fish so that after many trials it learns to choose the correct one. If $f(t)$ is the probability of choosing the rewarding gate after many trials (learned behavior), the mathematical model for this has the form of a functional equation
\begin{equation}\label{ec1}
	f(t)=tf(\alpha t+1-\alpha)+(1-t)f(\beta t), \quad f(0) = 0, \quad f(1) = 1, 
\end{equation}
for any $t\in [0,1]$. Here, $f:[0,1]\to \mathbb{R}$ is an unknown function and $0<\alpha\leq \beta <1 $. In Theorems 5.1 and 5.2 of \cite{turab2019analytic}, it is proved that under the assumption that $4\beta <1$, there exists a unique solution to (\ref{ec1}) in the complete metric space $CL^{0,1}=\{f\in H^1[0,1]:\ f(0)=0,\ f(1)=1\}$. For example, when $\alpha = 0$, that is, the fish always chooses the unrewarding gate, by iteration we have
\begin{equation}\label{ec15}
	f(t) = t f(1) + (1-t) f(\beta t) = t+ (1-t) f(\beta t).
\end{equation}
By inspection, one can show that the solution has the form
\begin{equation}
	f(t) = 1-\prod_{n=0}^{\infty}(1-\beta^n t),
\end{equation}
since, by a change of the variable $n+1 \mapsto n$, we have $(1-t)f(\beta t) = 1-t - (1-t)\prod_{n=0}^{\infty}(1-\beta^{n+1} t) =  1-t - \prod_{n=0}^{\infty}(1-\beta^{n} t) = f(t)-t$ which is equivalent to \eqref{ec15}. It is obvious that the occurrence of analytical closed-form solutions is an exceptional phenomenon, and in practice one is forced to use efficient numerical methods. 

Mathematical analysis of \eqref{ec1} or similar functional equations arising in behavioral science has been conducted in several other papers, in addition to those mentioned above. For example, in \cite{lyubich1973functional} Schauder's fixed point theorem was used to prove the existence under the assumption that the solution can be expanded in a certain power series. On the other hand, the Banach contraction theorem was used to show the existence and uniqueness in \cite{istruactescu1976functional}. Further advances have been achieved in \cite{dmitriev1982functional, berinde2015functional}. In our previous work \cite{paradiseFefiKishinLukas} we have considered a naturally generalized version of \eqref{ec1} and proved its unique solvability in the Lipschitz space by carefully choosing a closed subset and finding a suitable contractive operator. We have also observed that Picard's iteration, suggested by many authors as a method of obtaining approximate solutions, is extremely demanding on computational power. This made it difficult for practitioners to use. To remedy that, we provided some accurate and analytical approximations.  

In the present paper, we will study the following functional equation with boundary conditions
\begin{equation}\label{ec2}
	f(t)=\varphi(t)f(\varphi_1(t))+(1-\varphi(t))f(\varphi_2(t)), \quad f(0)=0, \quad f(1)=1,
\end{equation}
for any $t\in [0,1]$, where $f:[0,1]\to \mathbb{R}$ is a unknown function. The coefficients satisfy
\begin{equation}
	\begin{cases}
		\varphi:[0,1]\to \mathbb{R}, & \\
		\varphi_i:[0,1]\to [0,1], & \varphi_1(1)=1, \; \varphi_2(0)=0. \\
	\end{cases}
\end{equation}
Note that this setting generalizes \eqref{ec1} in a natural way in which all coefficients retain the essential properties of those in \eqref{ec1}. In the following, we prove that the solution to the above problem is unique and belongs to the space of H\"older function. As will be clear from the presentation, the choice of the function space is crucial for the relevant operator to become a contraction. H\"older space provides minimal regularity, apart from the lone continuity, to grant a unique solution. Previously, we were able to obtain an analogous result for Lipschitz functions \cite{paradiseFefiKishinLukas}. However, in the present work we show that it is possible to relax the regularity of the function space even more and still have a unique solution. Moreover, in the second part of the paper, we construct an efficient numerical scheme based on the collocation method and prove its convergence along with the order. This approach of finding approximate solutions is far superior to Picard's iteration, which was reported and verified in \cite{paradiseFefiKishinLukasCollo}. As a general rule, the convergence, both in practice and in the proof of it, is easier to obtain when sufficient regularity is available. The favorable case for piecewise linear collocation is having a twice-differentiable solution. For H\"older continuous functions showing that the numerical scheme is convergent requires more refined techniques. In the following, we show how to prove this result by using several estimates of the projection operator. These are auxiliary results, but may be interesting on their own. 

Equation \eqref{ec2} can be understood from several points of view. Note that the functions $\varphi_{1,2}$ mix the argument of the solution in a possibly nonlinear way. Therefore, the value of $f$ in $t$ always depends on the values of the function at $\varphi_1(t)$ and $\varphi_2(t)$. This makes the problem \textit{nonlocal}. Due to this complication, it would be difficult to use numerical methods based on a prescribed mesh. Rather, it is more natural to use schemes that return a function defined over the whole domain $[0,1]$ such as the collocation method \cite{brunner2004collocation}. This mixing of arguments introduces various difficulties that have been investigated by many authors. For example, in the paradise fish equation, the value of $f$ at $t$ to the values at $\alpha t + 1-\alpha$ and $\beta t$. These can be understood as a \textit{proportional} or, in general, a \textit{vanishing} delay (in contrast to the usual delay equations with terms being functions of $t - \tau$ for some $\tau > 0$). For example, the simplest of such problems $f(t) = a f (b t) + g(t)$ has gained considerable attention as a prototype functional equation associated with the celebrated pantograph equation \cite{iserles1993generalized, buhmann1993stability, yang2019mean}. In \cite{liu1995linear} existence and uniqueness of the solution has been proved along with the construction of a collocation scheme. These results were further generalized in \cite{brunner2011analysis, tian2020analysis} in which the convergence of the numerical method was proved. In some sense, the equation studied by us is more complex since it involves two independent ways of mixing arguments in a not necessarily proportional way. Furthermore, we consider a boundary value problem rather than an initial one in which the word "delay" loses its physical meaning. Rather, we would like to think of \eqref{ec2} as a nonlocal equation.

In the next section, we revise some elementary properties of H\"older spaces. Section 3 contains our main result concerning the existence and uniqueness, while in Section 4 we construct and analyze the collocation scheme. In Section 5 we illustrate our finding by some numerical simulations. We close the paper with some remarks on open problems that appeared during our work. 

\section{H\"older space preliminaries}
Before to present our main results, we will give the mathematical tools which we will use in the present paper. This material can be found in \cite{banas2013space}. Let $[a,b]$ be a closed interval in $\mathbb{R}$, $C[a,b]$ denote the space of continuous functions with real values in $[a,b]$ equipped with the classical supremum norm, that is, for $f\in C[a,b]$, $\|f\|_{\infty}=\sup \{|f(t)|:\ t\in [a,b]\}$. Similarly, the space of smooth functions with $k$-th continuous derivatives is denoted by $C^k[a,b]$. Now, we define two specific subsets of the space of continuous functions. For $0<\gamma\leq 1$ fixed, $H^\gamma[a,b]$ will denote the space of real functions $f:[a,b]\to \mathbb{R}$ such that 
\begin{equation}
	\sup \left\{\frac{|f(t)-f(s)|}{|t-s|^\gamma}:\ t,s\in [a,b],\ \ t\neq s\right\}<\infty.
\end{equation}
It is easily seen that $H^\gamma[a,b]$ is a linear subspace of $C[a,b]$ and can be normed by 
\begin{equation}
	\|f\|_{\gamma}=|f(a)|+\sup \left\{\frac{|f(t)-f(s)|}{|t-s|^\gamma}:\ t,s\in [a,b],\ t\neq s\right\},
\end{equation} 
for any $f\in H^\gamma[a,b]$. It is proved that $(H^\gamma[a,b], \|\cdot\|_{\gamma})$ is a Banach space. The space $H^\gamma[a,b]$ is called the \textbf{H\"older space} and in the specific case $\gamma=1$, the \textbf{Lipschitz space}. A closed subspace of the H\"older space will also be of use: functions vanishing at the boundary, that is,
\begin{equation}
	H_0^\gamma[a,b] := \{f \in H^\gamma[a, b]: \; f(a) = f(b) = 0\},
\end{equation}
which is a Banach space with the inherited norm $\|\cdot\|_\gamma$ (note that the term $f(a) = 0$).

For numerical methods, we sometimes also require some more regularity and define the space of functions of which derivatives are H\"older continuous. The following space will be useful in studying the error of our numerical method
\begin{equation}
	H^{k,\gamma}[a,b] := 
	\begin{cases}
		H^\gamma[a,b], & k = 0, \\
		\{f:[a,b]\mapsto \mathbb{R}:\; f' \in H^\gamma[a,b]\}, & k = 1,
	\end{cases}
\end{equation}
where prime denotes the derivative. The space $H_0^{k,\gamma}[a,b]$ is defined analogously. It is also clear how to define H\"older spaces of higher smoothness, however, we will not need them. We can state some simple properties of the H\"older spaces $H^\gamma[a,b]$.
\begin{lemma}[\cite{1n}]\label{lema1}
	For $f\in H^\gamma[a,b]$, the following inequality holds
	\begin{equation}
		\|f\|_{\infty}\leq \max\{1, (b-a)^{\gamma}\} \|f\|_{\gamma}.
	\end{equation}
\end{lemma}
\begin{lemma}[\cite{1n}]\label{lema2}
	For $0<\gamma<\beta\leq 1$, we have
	\begin{equation}
		H^{\beta}[a,b]\subset H^\gamma[a,b]\subset C[a,b].
	\end{equation}
	Moreover, if $f\in H^{\beta}[a,b]$ then
	\begin{equation}
		\|f\|_{\gamma}\leq \max \{1, (b-a)^{\beta-\gamma}\} \|f\|_{\beta}.
	\end{equation}
\end{lemma}
Finally, it is known that $H^\gamma[a,b]$ is a Banach algebra, and to make the paper self-contained, we give the proof of this fact.
\begin{lemma}\label{lema3}
	For $f,g\in H^\gamma[a,b]$, we have $f\cdot g\in H^\gamma[a,b]$.
\end{lemma}
\begin{proof}
	We take $t,s\in [a,b]$ with $t\neq s$, and then since $f,g\in H^\gamma[a,b]$, it follows that
	\begin{equation}
		\begin{split}
			\frac{|f(t)g(t)-f(s)g(s)|}{|t-s|^{\gamma}}&\leq	\frac{|f(t)g(t)-f(s)g(t)|}{|t-s|^{\gamma}}+\frac{|f(s)g(t)-f(s)g(s)|}{|t-s|^{\gamma}}\\
			&\leq |g(t)|\frac{|f(t)-f(s)|}{|t-s|^{\gamma}}+|f(s)|\frac{|g(t)-g(s)|}{|t-s|^{\gamma}}\\
			&\leq \|g\|_{\infty}\sup \left\{\frac{|f(p)-f(q)|}{|p-q|^{\gamma}}:\ p,q\in [a,b], \ p\neq q\right\}\\
			&+\|f\|_{\infty}\sup \left\{\frac{|g(p)-g(q)|}{|p-q|^{\gamma}}:\ p,q\in [a,b], \ p\neq q\right\}<\infty,
		\end{split}
	\end{equation}
	which ends the proof.
\end{proof}

\begin{remark}\label{nota1}
	Taking into account the above estimate of Lemma \ref{lema3}, we infer that 
	\begin{equation}
		\|f\cdot g\|_{\gamma}\leq \|f\|_{\infty}(\|g\|_{\gamma}-|g(0)|)+\|g\|_{\infty}(\|f\|_{\gamma}-|f(0)|).
	\end{equation}
\end{remark}

From now on, we will consider the interval $[a,b]=[0,1]$ and the H\"older space $H^{\gamma}[0,1]$ for $0<\gamma\leq 1$. We proceed to prove some auxiliary results to prepare for the main result. 

\begin{lemma}\label{lema4}
	The identity mapping $I:[0,1]\to [0,1]$, defined by $I(t)=t$, belongs to $H^{\gamma}[0,1]$ for $0<\gamma\leq 1$.
\end{lemma}
\begin{proof}
	It is clear that, for any $t,s\in [0,1]$ with $t\neq s$, 
	\begin{equation}
		\frac{|I(t)-I(s)|}{|t-s|^{\gamma}}=\frac{|t-s|}{|t-s|^{\gamma}}=|t-s|^{1-\gamma}\leq 1,
	\end{equation}
	and this proves that $I\in H^{\gamma}[0,1]$.
\end{proof}
The following lemma is related to the composition operator in $H^{\gamma}[0,1]$.
\begin{lemma}\label{lema5}
	Suppose that $f\in H^{\gamma}[0,1]$ with $0<\gamma\leq 1$ and $\varphi:[0,1]\to [0,1]$ such that $\varphi \in H^1[0,1]$. Then the operator defined by 
	\begin{equation}
		C_{\varphi}(f)(t)=f(\varphi(t))\quad \text{for}\ t\in [0,1],
	\end{equation}
	satisfies:
	\begin{itemize}
		\item[(i)] $C_{\varphi}f\in H^{\gamma}[0,1]$,
		\item[(ii)] $\|C_{\varphi}f\|_{\gamma}\leq |f(\varphi(0))|+(\|f\|_{\gamma}-|f(0)|)\cdot (\|\varphi\|_1-|\varphi(0)|)^{\gamma}$,
		\item[(iii)] under the assumption $f(0)=0$, for any $t\in [0,1]$ we have
		\begin{equation}
			|C_{\varphi}f(t)|\leq \|f\|_{\gamma}\|\varphi\|_1^{\gamma}.
		\end{equation}
	\end{itemize}
\end{lemma}
\begin{proof}
	For the proof of $(i)$ we take $t,s\in [0,1]$ with $t\neq s$ and estimate the following quotient
	\begin{equation}
		\frac{|(C_{\varphi}f)(t)-(C_{\varphi}f)(s)|}{|t-s|^{\gamma}}.
	\end{equation}
	In fact, taking into account that $f\in H^{\gamma}[0,1]$ and $\varphi\in H^1[0,1]$, we have
	\begin{equation}
		\begin{split}
			\frac{|(C_{\varphi}f)(t)-(C_{\varphi}f)(s)|}{|t-s|^{\gamma}}&=\frac{|f(\varphi(t))-f(\varphi(s))|}{|t-s|^{\gamma}}\\
			&\leq \frac{|f(\varphi(t))-f(\varphi(s))|}{|\varphi(t)-\varphi(s)|^{\gamma}}\frac{|\varphi(t)-\varphi(s)|^{\gamma}}{|t-s|^{\gamma}}\leq \|f\|_\gamma \|\varphi\|_1^\gamma<\infty,
		\end{split}
	\end{equation}
	which proves that $C_{\varphi}f\in H^{\gamma}[0,1]$. Next, from the definition of $\|\cdot\|_{\gamma}$ and from the estimate obtained in (i), it further follows that
	\begin{equation}
		\|C_{\varphi}f\|_{\gamma}=|f(\varphi(0))|+(\|f\|_{\gamma}-|f(0)|)(\|\varphi\|_1 -|\varphi(0)|)^{\gamma}.
	\end{equation}
	which completes the proof of $(ii)$. For the claim $(iii)$, suppose that $f(0)=0$, then
	\begin{equation}
		\begin{split}
			|C_{\varphi}&f(t)|=|f(\varphi(t))|=|f(\varphi(t))-f(0)|\leq \frac{|f(\varphi(t))-f(0)|}{|\varphi(t)|^{\gamma}}|\varphi(t)|^{\gamma}\\
			&\leq \sup \left\{\frac{|f(p)-f(q)|}{|p-q|^{\gamma}}: \ p,q\in [0,1],\ p\neq q\right\}\|\varphi\|_\infty^{\gamma}\leq \|f\|_{\gamma}\|\varphi\|_1^{\gamma},
		\end{split}
	\end{equation}
	where we have used Lemma \ref{lema1}, specifically, $\|\varphi\|_{\infty}\leq \|\varphi\|_1$. The proof is complete. 
\end{proof}

\section{Existence and uniqueness}
In this section we will prove our main result concerning the existence and uniqueness of solutions to \eqref{ec2}. By $D^{0,1}_{\gamma}[0,1]$ we denote the following set
\begin{equation}
	D^{0,1}_{\gamma}[0,1]=\{f\in H^{\gamma}[0,1]:\ f(0)=0,\ f(1)=1\}.
\end{equation}
Note that, by Lemma \ref{lema1}, $\|f\|_{\infty}\leq \|f\|_{\gamma}$. Take $(f_n)\subset D^{0,1}_{\gamma}[0,1]$ and $f_n\stackrel{\|\cdot\|_{\gamma}}{\longrightarrow}f $ with $f\in H^{\gamma}[0,1]$. Then $f_n\stackrel{\|\cdot\|_{\infty}}{\longrightarrow}f$ and, therefore, $(f_n)$ converges pointwise to $f$, in particular $f_n(0)\to f(0)$ and $f_n(1)\to f(1)$. Hence, $f\in D_{\gamma}^{0,1}[0,1]$, which means that
$D_{\gamma}^{0,1}[0,1]$ is a closed subset of $H^{\gamma}[0,1]$. From this reasoning, the pair $(D_{\gamma}^{0,1}[0,1], d)$, where $d$ is the distance given by 
\begin{equation}
	d(f,g)=\|f-g\|_{\gamma}, \quad \text{for}\ f,g\in D_{\gamma}^{0,1}[0,1],
\end{equation}
is a complete metric space. Furthermore, let $T$ be the operator defined on $D_{\gamma}^{0,1}[0,1]$ by 
\begin{equation}\label{ec3}
	(Tf)(t)=\varphi(t)f(\varphi_1(t))+(1-\varphi(t))f(\varphi_2(t)),
\end{equation}
for any $t\in [0,1]$. The following results state some fundamental properties of $T$ acting on $D_\gamma^{0,1}[0,1]$.

\begin{theorem}\label{teo1}
	Let $0<\gamma\leq 1$ be. Suppose that $f\in D_{\gamma}^{0,1}[0,1]$, $\varphi:[0,1]\to \mathbb{R}$ with $\varphi\in D_{\gamma}^{0,1}[0,1]$ and 
	$\varphi_1,\varphi_2:[0,1]\to [0,1]$, $\varphi_1 ,\varphi_2 \in H^1 [0,1]$ with $\varphi_1(1)=1$ and $\varphi_2(0)=0$. Then
	\begin{itemize}
		\item[(i)] $Tf\in D_{\gamma}^{0,1}[0,1]$.
		\item[(ii)] $\|Tf\|_{\gamma}\leq \|\varphi\|_{\gamma}(2\|\varphi_2\|_1^{\gamma}+(\|\varphi_1\|_1-\varphi_1(0))^{\gamma}+\|\varphi_1\|^{\gamma})\|f\|_{\gamma}$.
		\item[(iii)] For $f,g\in D_{\gamma}^{0,1}[0,1]$, we have
		\begin{equation}
			d(Tf, Tg)\leq  2\|\varphi\|_{\gamma}(\|\varphi_2\|_1^{\gamma}+(\|\varphi_1\|_1-\varphi_1(0))^{\gamma})d(f,g).
		\end{equation}
		\item[(iv)] If $ \|\varphi\|_{\gamma}(2\|\varphi_2\|_1^{\gamma}+(\|\varphi_1\|_1-\varphi_1(0))^{\gamma}+\|\varphi_1\|^{\gamma})<1$, then the operator $T$ given by (\ref{ec3}) has a unique fixed point $f^{\star}$ in $D_{\gamma}^{0,1}[0,1]$. Moreover, given $f_0\in D_{\gamma}^{0,1}[0,1]$ the iteration $\{f_n\}$ in $D_{\gamma}^{0,1}[0,1]$ defined by
		\begin{equation}
			f_n(t)=\varphi(t)f_{n-1}(\varphi_1(t))+(1-\varphi(t))f_{n-1}(\varphi_2(t))
		\end{equation}
		for $n\in \mathbb{N}$, converges to the unique solution $f^{\star}$.
	\end{itemize}
\end{theorem}
\begin{proof}
	We start with proving $(i)$. Notice that the operator $T$ can be expressed by
	\begin{equation}
		Tf= \varphi\cdot C_{\varphi_1}(f)+(1-\varphi)\cdot C_{\varphi_2}(f),
	\end{equation}
	and taking into account Lemmas \ref{lema3}, \ref{lema4} and \ref{lema5}, it follows that for $f\in D_{\gamma}^{0,1}[0,1]\subset H^{\gamma}[0,1]$ we have $Tf\in H^{\gamma}[0,1]$. Moreover, we see that 
	\begin{equation}
		(Tf)(0)=\varphi(0)f(\varphi_1(0))+(1-\varphi(0))f(\varphi_2(0))=f(\varphi_2(0))=f(0)=0
	\end{equation}
	where we have used the fact that $\varphi, f\in D_{\gamma}^{0,1}[0,1]$ and our assumption $\varphi_2(0)=0$. On the other hand, since $f,\varphi\in  D_{\gamma}^{0,1}[0,1]$ and $\varphi_1(1)=1$, we deduce
	\begin{equation}
		(Tf)(1)=\varphi(1)f(\varphi_1(1))+(1-\varphi(1))f(\varphi_2(1))=f(\varphi_1(1))=f(1)=1.
	\end{equation}
	This proves that $Tf\in D_{\gamma}^{0,1}[0,1]$.
	
	For the proof of $(ii)$ we have to estimate $\|Tf\|_{\gamma}$ for $f\in D_{\gamma}^{0,1}[0,1]$, we take $t,s\in [0,1]$ with $t\neq s$ and we have
	\begin{equation}
		\begin{split}
			\frac{|(Tf)(t)-(Tf)(s)|}{|t-s|^{\gamma}}
			&=\frac{1}{|t-s|^{\gamma}}\left[|\varphi(t)f(\varphi_1(t))+(1-\varphi(t))f(\varphi_2(t))-\varphi(s)f(\varphi_1(s))-(1-\varphi(s))f(\varphi_2(s))|\right] \\
			&\leq \frac{1}{|t-s|^{\gamma}}\left[|\varphi(t)f(\varphi_1(t))-\varphi(t)f(\varphi_1(s))|\right.+|(1-\varphi(t))f(\varphi_2(t))-(1-\varphi(t))f(\varphi_2(s))|\\
			&+|\varphi(t)f(\varphi_1(s))-\varphi(s)f(\varphi_1(s))|+\left. |(1-\varphi(t))f(\varphi_2(s))-(1-\varphi(s))f(\varphi_2(s))|\right].
		\end{split}
	\end{equation}
	Now, we can form the respective quotients
	\begin{equation}
		\begin{split}
			\frac{|(Tf)(t)-(Tf)(s)|}{|t-s|^{\gamma}}
			&\leq \frac{1}{|t-s|^{\gamma}}\left[|\varphi(t)|\frac{|f(\varphi_1(t))-f(\varphi_1(s))|}{|\varphi_1(t)-\varphi_1(s)|^{\gamma}}|\varphi_1(t)-\varphi_1(s)|^{\gamma} \right.\\
			&+ |1-\varphi(t)|\frac{|f(\varphi_2(t))-f(\varphi_2(s))|}{|\varphi_2(t)-\varphi_2(s)|^{\gamma}}|\varphi_2(t)-\varphi_2(s)|^{\gamma}\\
			& \left.+|\varphi(t)-\varphi(s)||f(\varphi_1(s))|+|\varphi(t)-\varphi(s)||f(\varphi_2(s))|\right]
		\end{split}
	\end{equation}
	or
	\begin{equation}
		\begin{split}
			\frac{|(Tf)(t)-(Tf)(s)|}{|t-s|^{\gamma}} 
			&\leq \|\varphi\|_{\gamma}\frac{|f(\varphi_1(t))-f(\varphi_1(s))|}{|\varphi_1(t)-\varphi_1(s)|^{\gamma}}\left(\frac{\varphi_1(t)-\varphi_1(s)|}{|t-s|}\right)^{\gamma}\\
			&+ \frac{|\varphi(t)-\varphi(1)|}{|1-t|^{\gamma}}|1-t|^{\gamma}\frac{|f(\varphi_2(t))-f(\varphi_2(s))|}{|\varphi_2(t)-\varphi_2(s)|^{\gamma}}\left(\frac{|\varphi_2(t)-\varphi_2(s)|}{|t-s|}\right)^{\gamma}\\
			& +\frac{|\varphi(t)-\varphi(s)|}{|t-s|^{\gamma}}|f(\varphi_1(s))|+\frac{|\varphi(t)-\varphi(s)|}{|t-s|^{\gamma}}|f(\varphi_2(s))|,
		\end{split}
	\end{equation}
	where we have used the fact that $\varphi\in D_{\gamma}^{0,1}[0,1]$ and $\|\varphi\|_{\infty}\leq \|\varphi\|_{\gamma}$. Now, taking into account that \begin{equation}
		\frac{\varphi(t)-\varphi(1)}{|1-t|^{\gamma}}\leq \|\varphi\|_{\gamma}
	\end{equation}
	along with $\varphi_2(0)=0$ and
	\begin{equation}
		\frac{|f(\varphi_i(t))-f(\varphi_i(s))|}{|\varphi_i(t)-\varphi_i(s)|^{\gamma}}\leq \|f\|_{\gamma}, \quad \frac{|\varphi_i(t)-\varphi_i(s)|}{|t-s|}\leq \|\varphi_i\|_1-\varphi_i(0) \quad \text{for}\ i=1,2,
	\end{equation}
	we arrive at 
	\begin{equation}\label{eqn:ecnueva}
		\begin{split}
			\frac{|(Tf)(t)-(Tf)(s)|}{|t-s|^{\gamma}}&\leq \|\varphi\|_{\gamma}\|f\|_{\gamma}(\|\varphi_1\|_1-\varphi_1(0))^{\gamma}\\
			&+ \|\varphi\|_{\gamma}\|f\|_{\gamma}\|\varphi_2\|_1^{\gamma}+\|\varphi\|_{\gamma}|f(\varphi_1(s))|+\|\varphi\|_{\gamma}|f(\varphi_2)(s)|.
		\end{split}
	\end{equation}
	Now, by Lemma \ref{lema5} $(iii)$, we have
	\begin{equation}
		|f(\varphi_i(s))|\leq \|f\|_{\gamma}\|\varphi_i\|_1^{\gamma}  \quad \text{for}\ i=1,2.
	\end{equation}
	From this and the last estimate, we infer that
	\begin{equation}
		\begin{split}
			\frac{|(Tf)(t)-(Tf)(s)|}{|t-s|^{\gamma}}&\leq \|\varphi\|_{\gamma}\|f\|_{\gamma}(\|\varphi_1\|_1-\varphi_1(0))^{\gamma}+ \|\varphi\|_{\gamma}\|f\|_{\gamma}\|\varphi_2\|_1^{\gamma}+\|\varphi\|_{\gamma}|\|f\|_{\gamma}\|\varphi_1\|_1^{\gamma}\\
			&+\|\varphi\|_{\gamma}\|f\|_{\gamma}\|\varphi_2\|_1^{\gamma} = \|\varphi\|_{\gamma}\left(2\|\varphi_2\|_1^{\gamma} +(\|\varphi_1\|_1-\varphi_1(0))^{\gamma}+\|\varphi_1\|_1^{\gamma}\right)\|f\|_{\gamma}.
		\end{split}
	\end{equation}
	This inequality together with the fact that $(Tf)(0)=0$ give us 
	\begin{equation}
		\|Tf\|_{\gamma}\leq \|\varphi\|_{\gamma}\left(2\|\varphi_2\|_1^{\gamma} +(\|\varphi_1\|_1-\varphi_1(0))^{\gamma}+\|\varphi_1\|_1^{\gamma}\right)\|f\|_{\gamma},
	\end{equation}
	which proves $(ii)$.
	
	Next, we turn to $(iii)$. Taking into account the inequality \eqref{eqn:ecnueva} and since for any $f,g\in D_{\gamma}^{0,1}[0,1]$, $Tf-Tg=T(f-g)$, we deduce that, for fixed $t,s\in [0,1]$, $t\neq s$ 
	\begin{equation}\label{ecnueva2}
		\begin{split}
			\frac{|Tf(t)-Tg(t)-(Tf(s)-Tg(s))|}{|t-s|^{\gamma}}&= \frac{|T(f-g)(t)-T(f-g)(s)|}{|t-s|^{\gamma}} \\
			&\leq \|\varphi\|_{\gamma}\|f-g\|_{\gamma}(\|\varphi_1\|_1-\varphi_1(0))^{\gamma}
			+ \|\varphi\|_{\gamma}\|f-g\|_{\gamma}\|\varphi_2\|_1^{\gamma}\\
			&+\|\varphi\|_{\gamma}|(f-g)(\varphi_1(s))|+\|\varphi\|_{\gamma}|(f-g)(\varphi_2)(s)|.
		\end{split}
	\end{equation}
	Now, we will obtain an estimate of $|(f-g)(\varphi_1(s))|$. Since $(f-g)(1)=0$ and $\varphi_1(1)=1$, we deduce that 
	\begin{equation}
		\begin{split}
			|(f-g)(\varphi_1(s))|&= |(f-g)(\varphi_1(s))-(f-g)(1)|=\frac{|(f-g)(\varphi_1(s))-(f-g)(\varphi_1(1))|}{|\varphi_1(s)-\varphi_1(1)|^{\gamma}}|\varphi_1(s)-\varphi_1(1)|^{\gamma}\\
			&\leq \|f-g\|_{\gamma}\left(\frac{|\varphi_1(s)-\varphi_1(1)|}{|s-1|}\right)^{\gamma}|1-s|^{\gamma}\leq \|f-g\|_{\gamma}(\|\varphi_1\|_1-\varphi_1(0))^{\gamma}.
		\end{split}
	\end{equation}
	Taking into account this result and $(iii)$ from Lemma \ref{lema5}, from inequality (\ref{ecnueva2}) it follows that 
	\begin{equation}\label{ecnueva3}
		\begin{split}
			&\frac{|Tf(t)-Tg(t)-(Tf(s)-Tg(s))|}{|t-s|^{\gamma}}\leq \|\varphi\|_{\gamma}\|f-g\|_{\gamma}(\|\varphi_1\|_1-\varphi_1(0))^{\gamma}
			+ \|\varphi\|_{\gamma}\|f-g\|_{\gamma}\|\varphi_2\|_1^{\gamma}\nonumber\\&+\|\varphi\|_{\gamma}\|f-g\|_{\gamma}(\|\varphi_1\|_1-\varphi_1(0))^{\gamma}+\|\varphi\|_{\gamma}\|\varphi_2\|_1^{\gamma}\|f-g\|_{\gamma}=2\|\varphi\|_{\gamma}(\|\varphi_2\|_1^{\gamma}+(\|\varphi_1\|_1-\varphi_1(0))^{\gamma})\|f-g\|_{\gamma}.
		\end{split}
	\end{equation}
	Finally, since $(Tf-Tg)(0)=0$, we deduce that 
	\begin{equation}
		\begin{split}
			d(Tf,Tg)&=\|Tf-Tg\|_{\gamma}\leq 2\|\varphi\|_{\gamma}(\|\varphi_2\|_1^{\gamma}+(\|\varphi_1\|_1-\varphi_1(0))^{\gamma})\|f-g\|_{\gamma}\\
			&= 2\|\varphi\|_{\gamma}(\|\varphi_2\|_1^{\gamma}+(\|\varphi_1\|_1-\varphi_1(0))^{\gamma})d(f,g).
		\end{split}
	\end{equation}
	This completes the proof of $(iii)$. 
	
	
	Lastly, to determine $(iv)$ we use the Banach contraction principle along with the assumption that
	\begin{equation}
		\|\varphi\|_{\gamma}\left(2\|\varphi_2\|_1^{\gamma} +(\|\varphi_1\|_1-\varphi_1(0))^{\gamma}+\|\varphi_1\|_1^{\gamma}\right)<1,
	\end{equation}
	which, by $(iii)$ grants that $T$ is a contraction having a unique fixed point. 
\end{proof}
As we comment on in Section 1, the functional equation 
\begin{equation}\label{ec4}
	f(t)=tf(\alpha t+1-\alpha)+(1-t)f(\beta t)
\end{equation}
for any $t\in [0,1]$ is treated in \cite{turab2019analytic} under conditions $f(0)=0$, $f(1)=1$ and $0\leq \alpha<\beta<1$. This functional equation is a particular case of (\ref{ec2}), where $\varphi(t)=t$, $\varphi_1(t)=\alpha t +1-\alpha$, $\varphi_2(t)=\beta t$ for any $t\in [0,1]$. It can be easily seen that for $0<\gamma\leq 1$, $\varphi\in D_{\gamma}^{0,1}[0,1]$ and $\|\varphi\|_{\gamma}=1$ we have $\varphi_i:[0,1]\to [0,1],\ i=1,2$, $\varphi_1(1)=1$, $\varphi_2(0)=0$, $\varphi_1,\varphi_2 \in H^1[0,1]$, $\|\varphi_1\|_1=1$, $\varphi_1(0)=1-\alpha$, and $\|\varphi_2\|_1=\beta$. We also have the following
\begin{equation}
	2\|\varphi\|_{\gamma}(\|\varphi_2\|_1^{\gamma}+(\|\varphi_1\|_1-\varphi_1(0))^{\gamma}) = 2 \left(\beta^\gamma +(1-1+\alpha)^\gamma\right) = 2 (\alpha^\gamma + \beta^\gamma),
\end{equation} 
and hence, by Theorem \ref{teo1} $(iii)-(iv)$ the following result follows.

\begin{corollary}
	Suppose that $0<\alpha\leq\beta\leq 1$.
	\begin{itemize}
		\item[(a)] If $\alpha^\gamma+\beta^{\gamma}<1/2$ then the functional equation (\ref{ec4}) has a unique solution $f^{\star}$ in $D_{\gamma}^{0,1}[0,1]$. Furthermore, given $f_0\in D_{\gamma}^{0,1}[0,1]$ the iteration $\{f_n\}\in D_{\gamma}^{0,1}[0,1]$ given by
		\begin{equation}
			f_n(t)=tf_{n-1}(\alpha t+1-\alpha)+(1-t)f_{n-1}(\beta t)
		\end{equation}
		for any $n\in \mathbb{N}$ converges to the unique solution $f^{\star}$.
		\item[(b)] If $0< \beta <4^{-\frac{1}{\gamma}}$ then (\ref{ec4}) satisfies the same conclusion as in (a).
	\end{itemize}
\end{corollary}

We conclude this section with the remark that the above results can be easily generalized to nonhomogeneous equations with homogeneous boundary conditions, which are more feasible to analyze numerically. Hence, let us put $g(t) = f(t) - t$. Then, $g(0) = g (1) = 0$ and $Tg = Tf - Tt = f - Tt$, hence
\begin{equation}
	g(t) = Tg(t) + Tt - t.
\end{equation}  
Notice that since $f \in D^{0,1}_\gamma[0,1]$ we now have $g\in H_0^\gamma[0,1]$. The existence and uniqueness of $f$ is then equivalent to the existence and uniqueness of $g$. We can generalize this idea and consider the following problem
\begin{equation}\label{eqn:NonhomogeneousEq}
	f(t) = Tf(t) + k(t), \quad t\in[0,1], \quad f(0) = f(1) = 0, \quad k \in H_0^\gamma[0,1]. 
\end{equation}
It is easy to see that the above has a unique solution in $H_0^\gamma[0,1]$. 
\begin{corollary}\label{cor:Nohomogeneous}
	Under the assumptions of Theorem \ref{teo1} the problem \eqref{eqn:NonhomogeneousEq} has a unique solution in $H_0^\gamma[0,1]$. 
\end{corollary}
\begin{proof}
	Since $H_0^\gamma[0,1]$ is a Banach space, the operator $T$ has a well-defined norm which, due to the assumption, satisfies $\|T\| < 1$. Therefore, by the geometric series theorem, the operator $(I-T)^{-1}: H_0^\gamma[0,1] \mapsto H_0^\gamma[0,1]$ exists and has a trivial kernel (see \cite{zeidler2012applied}, Chapter 1.23). The unique solution is then given by $f = (I-T)^{-1} k$. 
\end{proof}

\section{Collocation method}
As was shown in our previous works \cite{paradiseFefiKishinLukasCollo,paradiseFefiKishinLukas} the iteration of Theorem \ref{teo1} is not necessarily the method of choice for the practical solution of functional equations. This is due to exponential complexity in both memory and CPU time. In \cite{paradiseFefiKishinLukasCollo} we have overcome this difficulty by devising a second-order collocation method, and in this section we extend it to the H\"older case. Note that numerical schemes usually behave well when applied to equations with solutions of a certain smoothness. In the low regularity case, some problems in proving convergence may arise. 

For generality, we solve \eqref{eqn:NonhomogeneousEq} instead of \eqref{ec1}. This also has the advantage that in this case we work in the Banach space $H_0^\gamma[0,1]$. The idea behind the collocation method \cite{brunner2004collocation} is based on the piecewise polynomial approximation on subintervals. More specifically, divide $[0,1]$ into $N$ subintervals. That is, $t_i = i h$ with $h = 1/N > 0$ where $i=0,1,...,N$. We construct a \emph{continuous} approximation $f_h$ to the solution of \eqref{ec1} by requiring that on each subinterval it is a \emph{linear polynomial} (a higher-order approximation can also be constructed on a refined grid). The linear function is then required to satisfy the functional equation at each node (it collocates the solution). Therefore, we have the following conditions for the collocation approximation
\begin{equation}\label{eqn:Collocation}
	\begin{cases}
		f_h(0) = 0, \; f_h(1) = 1, & \text{(boundary conditions)}, \\
		f_h(t_{i-1}^+) = f_h(t_i^-), \quad i = 1,2,..., N-1, & \text{(continuity)}, \\
		f_h(t_i) = Tf_h(t_i) + k(t_i), \quad i = 1,2,..., N-1, & \text{(collocation)}.
	\end{cases}
\end{equation}
Since in each interval the linear approximation has the form $a_i t + b_i$ in total we have $2N$ unknowns $\{a_i, b_i\}_{1}^N$ to determine. On the other hand, \eqref{eqn:Collocation} gives us $2$ boundary conditions, $N-1$ continuity points, and $N-1$ collocation equations. In total we have $2 + N - 1 + N - 1 = 2N$ equation for $2N$ unknowns and we can hope that the corresponding system can be uniquely solved. We will not pursue here the task of showing that the system matrix is non-singular. Due to nonlocality and the fact that \eqref{ec1} mixes arguments in a nonlinear way with $\varphi_{1,2}$ functions, this seems to be a difficult and interesting problem. A much simpler case was solved in \cite{brunner2011analysis} where the authors considered the pantograph functional equation with only one proportional delay. Their analysis required very involved techniques and consideration of many cases. We leave the corresponding problem for future work and focus only on the numerical aspects of the scheme.  

To analyze the collocation scheme, we introduce the linear projection operator by
\begin{equation}\label{eqn:Projection}
	P_h u(t) = \frac{1}{h} \left((t-t_{i-1})u(t_{i}) + (t_{i}-t)u(t_{i-1})\right), \quad t\in [t_{i-1}, t_{i}], \quad i=1,2,3,...,N.
\end{equation}
In \cite{paradiseFefiKishinLukasCollo} we have shown that $P_h u$ is Lipschitz continuous with a unit norm. Since a Lipschitz function is $\gamma$-H\"older for $0<\gamma\leq 1$ the operator $P_h: H^\gamma\mapsto H^\gamma$ is well-defined and below we will find an estimate for its norm when acting between these spaces. Many properties of linear interpolation are well known for the $C^2$ functions. For our analysis, we need more refined estimates for functions of much lower regularity. First, we prove the error bound in the supremum norm. 
\begin{proposition}\label{prop:ProjectionErrorSupremum}
	Let $u\in H^{k,\gamma}[0,1]$ with $k=0,1$ and $0<\gamma\leq 1$. Then,
	\begin{equation}
		\|P_h u - u\|_\infty \leq 2^{-\gamma-(2-\gamma)k} h^{k+\gamma} \|u\|_{k,\gamma}.
	\end{equation}
\end{proposition}
\begin{proof}
	For any $t\in [t_i, t_{i+1}]$ we have
	\begin{equation}\label{eqn:ProjectionErrorEst}
		|P_h u(t) - u(t)| = \frac{1}{h} \left|(t-t_i)(u(t_{i+1})-u(t)) + (t_{i+1}-t)(u(t_i)-u(t))\right|.
	\end{equation}
	Now, for $k=0$ we can use the H\"older continuity of $u$ to obtain
	\begin{equation}\label{eqn:Projection2}
		|P_h u(t) - u(t)| \leq \frac{\|u\|_\gamma}{h} \left( (t-t_i)(t_{i+1}-t)^\gamma + (t_{i+1}-t) (t-t_i)^\gamma \right). 
	\end{equation}
	It is a simple calculation to show that the right-hand side of the above, as a function of $t$, attains its maximum for $t=(t_{i+1}+t_i)/2 = (i+1/2)h$, hence
	\begin{equation}
		|P_h u(t) - u(t)| \leq \frac{\|u\|_\gamma}{h} \left(2^{-1-\gamma} h^{1+\gamma} + 2^{-1-\gamma} h^{1+\gamma}\right) = 2^{-\gamma}h^\gamma \|u\|_\gamma,
	\end{equation}
	which proves the case when $k=0$. Now, for $k=1$ we can refine the estimate in \eqref{eqn:ProjectionErrorEst}
	\begin{equation}
		|P_h u(t) - u(t)| = \frac{(t_{i+1}-t) (t-t_i)}{h} \left| \frac{u(t_{i+1})-u(t)}{t_{i+1}-t} - \frac{u(t)-u(t_i)}{t-t_i}\right|.
	\end{equation}
	From the mean-value theorem there exist $\xi,\eta\in (t_i, t_{i+1})$ such that
	\begin{equation}
		\begin{split}
			|P_h u(t) - u(t)| 
			&= \frac{(t_{i+1}-t) (t-t_i)}{h} \left| u'(\xi) -u'(\eta)\right|\leq \frac{h^2}{4h} h^\gamma \|u'\|_\gamma \\
			&= 2^{-2} h^{1+\gamma} \|u'\|_\gamma \leq 2^{-2}h^{1+\gamma} \|u\|_{1,\gamma},
		\end{split}
	\end{equation}
	and the proof is complete.
\end{proof}
As we can see, for $H^{1,1}[0,1]$ functions, that is, when the derivative is Lipschitz, we obtain the optimal second-order accuracy of linear interpolation. Classically, this result is proven for $C^2[0,1]$ functions that coincide almost everywhere with $H^{1,1}[0,1]$. In the twice-differentiable case, the error constant is equal to $2^{-3}$, while for our case it is $2^{-2}$. It is also possible to obtain interpolation error bounds in H\"older norms which confirms the usual folklore that the order of interpolation is equal to the degree of regularity minus the strength of the norm. 
\begin{proposition}\label{prop:ProjectionErrorHolder}
	Let $u\in H^{k,\beta}[0,1]$ with $k=0,1$ and $0<\beta\leq 1$. Then, for $0<\gamma\leq\min\{1, k+\beta\}$ we have
	\begin{equation}
		\|P_h u - u\|_\gamma \leq C h^{k+\beta-\gamma} \|u\|_{k,\beta},
	\end{equation}
	where $C>0$ is a constant dependent only on $k$, $\beta$, and $\gamma$. 
\end{proposition}
\begin{proof}
	First, assume that $t,s \in [t_{i-1}, t_i]$. Set $e_h := P_h u - u$ and by \eqref{eqn:ProjectionErrorEst} write
	\begin{equation}
		\begin{split}
			|e_h(t) - e_h(s)| 
			&= \frac{1}{h} \left|\underline{(t-t_{i-1})(u(t_i)-u(t))}+\dashuline{(t_i-t)(u(t_{i-1})-u(t))}\right. \\
			&\left.- \underline{(s-t_{i-1})(u(t_i)-u(s))}-\dashuline{(t_i-s)(u(t_{i-1})-u(s))}\right|.
		\end{split}
	\end{equation}
	We can now combine the underlined terms of like type with each other
	\begin{equation}
		\begin{split}
			(t-t_{i-1})(u(t_i)-u(t)) &- (s-t_{i-1})(u(t_i)-u(s)) = (t-t_{i-1})(u(t_i)-u(t)) - (s-t_{i-1})(u(t_i)-u(t)) \\
			&+ (s-t_{i-1})(u(t_i)-u(t))- (s-t_{i-1})(u(t_i)-u(s)) \\
			&=(t-s)(u(t_i)-u(t)) - (s-t_{i-1}) (u(t) - u(s)),
		\end{split}
	\end{equation}
	and similarly with the dashed term
	\begin{equation}
		(t_i-t)(u(t_{i-1})-u(t)) - (t_i-s)(u(t_{i-1})-u(s)) = (s-t)(u(t_{i-1})-u(t)) - (t_i-s) (u(t) - u(s)).
	\end{equation}
	Therefore, we obtain a simple fundamental expression for the difference in error
	\begin{equation}\label{eqn:ProjectionErrorFundamental}
		|e_h(t) - e_h(s)| = \frac{1}{h} \left|(t-s) (u(t_i) - u(t_{i-1})) -(t_i-t_{i-1}) (u(t)-u(s))\right|.
	\end{equation}
	Now, for $k=0$ we can bound the above terms to obtain
	\begin{equation}
		|e_h(t) - e_h(s)| \leq \frac{\|u\|_\beta}{h}\left(|t-s|h^\beta + h |t-s|^\beta\right).
	\end{equation}
	Therefore, by dividing by $|t-s|^\gamma$ we have
	\begin{equation}\label{eqn:ProjectionErrorHolder0}
		\frac{|e_h(t) - e_h(s)|}{|t-s|^\gamma} \leq \frac{\|u\|_\beta}{h} \left(|t-s|^{1-\gamma} h^\beta + |t-s|^{\beta-\gamma} h\right) \leq 2 h^{\beta-\gamma}\|u\|_\beta, \quad t,s \in [t_{i-1}, t_i],
	\end{equation}
	because $|t-s|\leq h$. We have thus found the error estimate in the subinterval.
	
	Assume now that $k=1$ and write \eqref{eqn:ProjectionErrorFundamental} in the following form
	\begin{equation}
		|e_h(t) - e_h(s)| = |t-s| \left|\frac{u(t_i) - u(t_{i-1})}{t_{i}-t_{i-1}} - \frac{u(t)-u(s)}{t-s} \right|.
	\end{equation}
	From the mean-value theorem there exist $\xi,\eta \in (t_{i-1}, t_i)$ such that
	\begin{equation}\label{eqn:ProjectionErrorHolder1}
		\frac{|e_h(t) - e_h(s)|}{|t-s|^\gamma} = |t-s|^{1-\gamma} \left|u'(\xi) - u'(\eta)\right| \leq |t-s|^{1-\gamma} h^\beta \|u'\|_\beta \leq h^{1+\beta-\gamma} \|u\|_{1,\beta}, 
	\end{equation}
	for $t,s\in[t_{i-1}, t_i]$. To prove our assertion, we have to find similar bounds for any $t,s\in[0,1]$. Without loss of generality, suppose that $t \in [t_{i-1}, t_i]$ and $s \in [t_j, t_{j+1}]$ for $1 \leq i \leq j \leq N$. Then, noticing that on the nodes the interpolation error vanishes, that is, $e_h(t_l) = 0$ for $0\leq l\leq N$, by \eqref{eqn:ProjectionErrorHolder0} and \eqref{eqn:ProjectionErrorHolder1} we can write
	\begin{equation}
		|e_h(t) - e_h(s)| \leq |e_h(t) - e_h(t_i)| + |e_h(t_j) - e_h(s)| \leq C_0 h^{k+\beta-\gamma} \|u\|_{k,\beta} \left((t_i-t)^\gamma + (s-t_j)^\gamma\right),
	\end{equation}
	for some $C_0 > 0$. Furthermore, we use the reverse H\"older inequality to obtain
	\begin{equation}
		(t_i-t) + (s-t_j) \geq \left((t_i-t)^\gamma+(s-t_j)^\gamma\right)^\frac{1}{\gamma} \left(1^\gamma+1^\gamma\right)^{1-\frac{1}{\gamma}},
	\end{equation}
	or
	\begin{equation}
		(t_i-t)^\gamma+(s-t_j)^\gamma \leq 2^{1-\gamma} \left((t_i-t) + (s-t_j)\right)^\gamma.
	\end{equation}
	Therefore, there exists a constant $C>0$ such that
	\begin{equation}
		|e_h(t) - e_h(s)| \leq C h^{k+\beta-\gamma} \|u\|_{k,\beta} |s-t_j+t_i - t|^\gamma. 
	\end{equation}
	The proof concludes by observing the fact that $s-t_j+t_i - t \leq s-t$, dividing, and taking the supremum. 
\end{proof}

We can now proceed to showing that $P_h$ is a bounded operator when acting on the H\"older space. The main difficulty in finding the norm is the fact that $P_h u$ is a piecewise linear function and we have to compare values of $P_h u$ at points belonging to possibly distant subintervals. The following result states a bound on the norm, which, as in our numerical computations and the beginning of the proof, is not optimal. However, it suffices to show convergence of the collocation scheme, and we leave the problem of sharpening of the estimate for future work. 
\begin{lemma}\label{lem:ProjectionNorm}
	Let $P_h: H^{\gamma}[0,1] \mapsto H^\gamma[0,1]$ with $0<\gamma<1$. Then
	\begin{equation}\label{eqn:ProjectionNorm}
		\|P_h\| \leq 1+2^{1-\gamma}.
	\end{equation}
\end{lemma}
\begin{proof}
	Take $t,s\in [0,1]$. We will consider three cases: $t$ and $s$ are in the same subinterval, in neighboring subintervals, or are separated by at least a subinterval. More specifically, we consider
	\begin{itemize}
		\item[(i)] $t,s\in [t_{i-1}, t_i]$ for some $i = 1,2,...,N$,
		\item[(ii)] $t\in [t_{i-1}, t_i]$ and $s\in [t_i, t_{i+1}]$ for some $i = 1,2,...,N-1$,
		\item[(iii)] $t\in [t_{i-1}, t_i]$ and $s\in [t_j, t_{j+1}]$ for some $i = 1, 2, ..., N-2$ and $j = 3,4, ..., N-1$ with $i < j$,
	\end{itemize}
	where, without any loss of generality, we have assumed that $t<s$. When $t$ and $s$ belong to a subinterval $[t_{i-1}, t_i]$, the function $P_h u$ is a linear segment that is trivially H\"older and according to the definition \eqref{eqn:Projection}
	\begin{equation}
		|P_h u(t) - P_h u(s)| = \frac{|u(t_i)-u(t_{i-1})|}{h} |t-s|,
	\end{equation}
	hence, since $P_h u(0) = u(0)$ and $|t-s|<h$ we have
	\begin{equation}\label{eqn:ProjectionNorm(i)}
		|P_h(0)| + \frac{|P_h u(t) - P_h u(s)|}{|t-s|^\gamma} \leq |u(0)| + \sup_{t\neq s} \frac{|u(t)-u(s)|}{|s-t|^\gamma} = \|u\|_\gamma,
	\end{equation}
	which proves (i).
	
	Now, we consider case (ii), that is, $t$ and $s$ are in neighboring subintervals with $t_{i-1} \leq t \leq t_{i} < s \leq t_{i+1}$. We have
	\begin{equation}
		|P_h u(t) - P_h u(s)| \leq |P_h u(t) - u(t_i)| + |u(t_i) - P_h u(s)|,
	\end{equation}
	and due to the definition of the projection \eqref{eqn:Projection} and $u(t_i) = h^{-1}((t-t_{i-1})+(t_i-t))u(t_i)$ we can write
	\begin{equation}\label{eqn:ProjectionErrorNode}
		|P_h u(t) - u(t_i)| = \frac{t_i-t}{h}|u(t_i)-u(t_{i-1})| \leq \left(\sup_{t\neq s} \frac{|u(t)-u(s)|}{|s-t|^\gamma} \right) h^{\gamma-1} (t_i-t),
	\end{equation}
	and similarly for $|P_h u(s) - u(t_i)|$. Then,
	\begin{equation}
		\begin{split}
			|P_h u(t) - P_h u(s)| 
			&\leq \left(\sup_{t\neq s} \frac{|u(t)-u(s)|}{|s-t|^\gamma} \right) h^{\gamma-1} (t_i-t+s-t_i) \\
			&=  \left(\sup_{t\neq s} \frac{|u(t)-u(s)|}{|s-t|^\gamma} \right) h^{\gamma-1} |t-s| \leq 2^{1-\gamma}\left(\sup_{t\neq s} \frac{|u(t)-u(s)|}{|s-t|^\gamma} \right)|t-s|^\gamma,
		\end{split}
	\end{equation}
	since in our case $|t-s|<2h$. Therefore, after dividing by $|t-s|^\gamma$, adding $P_h u(0) = u(0)$, and taking the supremum, we obtain 
	\begin{equation}\label{eqn:ProjectionNorm(ii)}
		|P_h(0)| + \frac{|P_h u(t) - P_h u(s)|}{|t-s|^\gamma} \leq |u(0)| + 2^{1-\gamma}\sup_{t\neq s} \frac{|u(t)-u(s)|}{|t-s|^\gamma} \leq 2^{1-\gamma} \|u\|_\gamma,
	\end{equation}
	which is the needed estimate for the (ii) case.
	
	
	For the (iii) case, we can again use the approximation property of the interpolation. That is, by Proposition \ref{prop:ProjectionErrorSupremum}
	\begin{equation}
		\begin{split}
			|P_h u(t) - P_h u(s)| 
			&\leq |P_h u(t) - u(t)| + |u(t) - u(s)| + |u(s) - P_h u(s)| \\
			&\leq 2^{1-\gamma} h^\gamma \|u\|_\gamma + \sup_{t\neq s} \frac{|u(t)-u(s)|}{|s-t|^\gamma} |t-s|^\gamma.
		\end{split}
	\end{equation}
	Now, since $h\leq |t-s|$ it follows that
	\begin{equation}\label{eqn:ProjectionNorm(iii)}
		\begin{split}
			|P_h u(0)| + \frac{|P_h u(t) - P_h u(s)|}{|t-s|^\gamma} &\leq 2^{1-\gamma} \|u\|_\gamma + |u(0)| + \sup_{t\neq s} \frac{|u(t)-u(s)|}{|s-t|^\gamma} \\
			&
			= (1+2^{1-\gamma}) \|u\|_\gamma.
		\end{split}
	\end{equation}
	Now, in estimates for each of the three cases (i)-(iii), that is \eqref{eqn:ProjectionNorm(i)}, \eqref{eqn:ProjectionNorm(ii)}, and \eqref{eqn:ProjectionNorm(iii)}, we can take the supremum over all $t\neq s$ to obtain that $\|P_h u\|_\gamma \leq (1+2^{1-\gamma})\|u\|_\gamma$. This implies $\|P_h\| \leq 1+2^{1-\gamma}$ and concludes the proof. 
\end{proof}

Now, we are ready to prove our main result in this section, the uniform convergence of the collocation scheme. From the proof, it can also be inferred that the method converges in the H\"older norm. However, we state our error bound in the supremum norm since it is much more natural and useful in practice. 
\begin{theorem}\label{thm:Convergence}
	Let $f$ be the solution of \eqref{ec1} while $f_h$ its collocation approximation that can be found from \eqref{eqn:Collocation}. Assume that $2\|\varphi\|_{\gamma}(\|\varphi_2\|_1^{\gamma}+(\|\varphi_1\|_1-\varphi_1(0))^{\gamma})<(1+2^{1-\gamma})^{-1}$ for some $0<\gamma<1$. Then, 
	\begin{equation}
		\|f-f_h\|_\infty \leq C h^{\gamma},
	\end{equation}
	for some $C>0$ dependent only on $f$, $\gamma$, and the coefficients of \eqref{ec1}. Moreover, if $f\in H_0^{1,\gamma}$, then
	\begin{equation}
		\|f-f_h\|_\infty \leq C h^{1+\gamma}.
	\end{equation}
\end{theorem}
\begin{proof}
	By the assumption and Corollary \ref{cor:Nohomogeneous} we know that there exists a unique solution $f\in H_0^\gamma[0,1]$. Start by noticing that since $f_h$ is a linear function, we can act with $P_h$ on the collocation equation \eqref{eqn:Collocation} and obtain
	\begin{equation}
		f_h = P_h T f_h + P_h k.
	\end{equation}
	We can also apply $P_h$ to the exact equation $f = Tf + k$ and subtract to arrive at
	\begin{equation}
		f_h - P_h f = P_h T f_h - P_h T f,
	\end{equation}
	or by subtracting $f$ we can write
	\begin{equation}
		f_h - f = P_h T(f_h - f) + P_h f - f
	\end{equation}
	This leads to
	\begin{equation}
		f - f_h = (I-P_h T)^{-1} (f-P_h f),
	\end{equation}
	since by the geometric series theorem, the operator $(I-P_h T)^{-1}$ exists because of our assumption. Moreover, Lemma \ref{lem:ProjectionNorm}, and Theorem \ref{teo1} $(iii)$ give
	\begin{equation}
		\|P_h T\| \leq \|P_h\| \|T\| \leq (1+2^{1-\gamma}) \left(2\|\varphi\|_{\gamma}(\|\varphi_2\|_1^{\gamma}+(\|\varphi_1\|_1-\varphi_1(0))^{\gamma})\right) < 1.
	\end{equation}
	When we apply the $\gamma$-H\"older norm, we can obtain the bound for the scheme error
	\begin{equation}
		\|f-f_h\|_\gamma \leq \|(I-P_h T)^{-1}\| \|f-P_h f\|_\gamma.
	\end{equation}
	Now, for $f\in H^{k,\gamma}[0,1]$ with $k=0,1$ by Proposition \ref{prop:ProjectionErrorHolder} we have 
	\begin{equation}\label{eqn:ConvergenceHolderNorm}
		\|f-f_h\|_\gamma \leq \frac{C}{1-\|P_h T\|} h^{k+\gamma-\gamma} \|u\|_{k,\gamma} \leq C_1 h^k,
	\end{equation}
	for some $h$-independent constant $C_1>0$. For any $t\in [0,1]$ there exists a subinterval such that $t\in [t_{i-1}, t_i]$. Then,
	\begin{equation}
		|f(t)-f_h(t)| \leq |f(t) - f(t_i)| + |f(t_i) - f_h(t)| = |f(t) - P_h f(t_i)| + |f(t_i) - f_h(t)|, 
	\end{equation}
	because the interpolant $P_h f$ is equal to $f$ on the nodes. Then, by the H\"older norm estimate \eqref{eqn:ConvergenceHolderNorm} and again by Proposition \ref{prop:ProjectionErrorHolder} we have
	\begin{equation}
		\begin{split}
			|f(t)-f_h(t)| 
			&\leq \|f - P_h f\|_\gamma h^\gamma + \|f-f_h\|_\gamma h^\gamma \leq C h^{k+\gamma-\gamma} h^\gamma \|f\|_{k,\gamma} + C_1 h^{k + \gamma} \\
			&\leq C h^{\gamma + k + \gamma-\gamma} = C h^{k+\gamma}.
		\end{split}
	\end{equation}
	Taking the supremum over $t\in[0,1]$ gives the conclusion. 
\end{proof}
As can be seen from the theorem above, the error of the numerical scheme is of the same order as the interpolation error. Minimally, the error is of order $\gamma$, but if the solution is more regular, it can increase to the optimal level. A discussion of the regularity of solutions to \eqref{eqn:NonhomogeneousEq} can be found in \cite{paradiseFefiKishinLukasCollo}. 

\section{Numerical example}
In this section, we present several illustrative examples that verify our above theoretical considerations. Since the case of smooth solutions has been analyzed in \cite{paradiseFefiKishinLukasCollo}, here we focus only on the less regular, H\"older continuous functions. 

It is a matter of choosing an appropriate $k$ in \eqref{eqn:NonhomogeneousEq} to produce an exact arbitrary solution. Thanks to that, we can test the collocation scheme in various situations. For example, we can choose the following $\gamma$-H\"older function
\begin{equation}\label{eqn:ExactFunction0}
	f(t) = \left(\frac{1}{2}-\left|t-\frac{1}{2}\right|\right)^\gamma \in H_0^\gamma[0,1], \quad 0<\gamma<1.
\end{equation}
Note that $f$ is not differentiable, has a characteristic cusp at $t=1/2$, and for small $\gamma$ its derivative at $t=0,1$ becomes unbounded (see Fig. \ref{fig:ExactFunction0}). 

\begin{figure}
	\centering
	\includegraphics{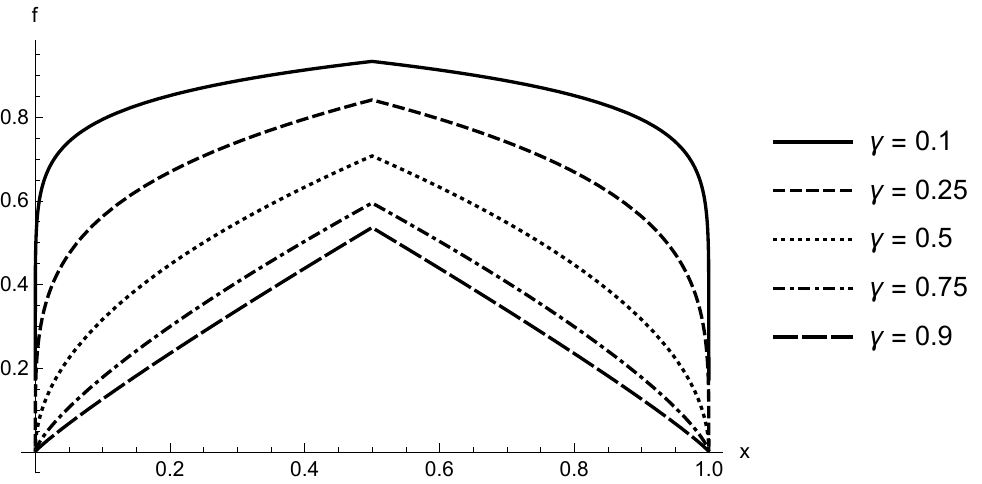}
	\caption{Exemplary exact solution \eqref{eqn:ExactFunction0} for different $\gamma$. }
	\label{fig:ExactFunction0}
\end{figure}

For the coefficients we choose
\begin{equation}
	\varphi(x) = x, \quad \varphi_1(x) = 1-\frac{\alpha}{2}(1-x), \quad \varphi_2(x) = \frac{\alpha}{2}x, \quad \alpha \in \left(0, (2^{2-\gamma}(1+2^{1-\gamma}))^{-\gamma}\right).
\end{equation}
By a straightforward computation we can compute the respective norms of the coefficients to find that $\|\varphi\|_\gamma = 1$, $\|\varphi_1\|_1 = 1$, $\|\varphi_2\|_1 = \alpha/2$. Hence, with our choice of the range for the parameter $\alpha$ we check that the assumption of Theorem \ref{thm:Convergence} is satisfied. We also tested other choices with essentially the same conclusions. The maximum error of the collocation scheme, that is $\|u-u_h\|_\infty$ is depicted in Fig. \ref{fig:Error0} on a log-log scale. As can be seen, the computations verify the claim of Theorem \ref{thm:Convergence} that the order of convergence is strongly related to the smoothness of the solution and is equal to $\gamma$. Note that error lines become almost parallel to the reference even for a small number of subintervals. 

\begin{figure}
	\centering
	\includegraphics{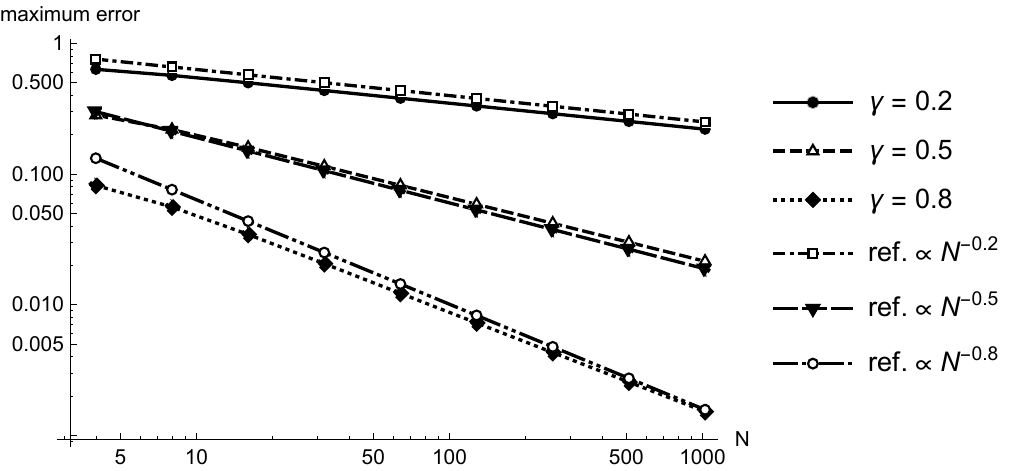}
	\caption{Maximum error of the collocation scheme applied \eqref{eqn:NonhomogeneousEq} with the exact solution \eqref{eqn:ExactFunction0}. }
	\label{fig:Error0}
\end{figure}

\section{Conclusion and look forward}
Functional equations with proportional delay pose an interesting problem for both theoretical and numerical analysis. A careful choice of the function space in which the solution is to be sought allows for the use of the Banach contraction principle to prove the existence and uniqueness. From a numerical point of view, folklore tells us that the more regularity the solution has, the better the numerical approximation. In the H\"older case, however, we are far from the optimal twice-differentiable situation, which requires the development of different techniques to find the interpolation error. From that, the convergence proof follows the usual path.

We have encountered two interesting open problems to investigate in future work. On the one hand, in this paper, the solution of the algebraic system \eqref{eqn:Collocation} is assumed to exist. A proof of this claim would be most welcome for the completeness of the theory. However, as mentioned above, the development of the general $p$-order collocation scheme would produce a very versatile and accurate numerical method. We plan to address both of these problems in future work. 

\section*{Acknowledgements}
\L.P. has been supported by the National Science Centre, Poland (NCN) under the grant Sonata Bis with a number NCN 2020/38/E/ST1/00153. J.C. and K. S. are partially supported by the project PID2023-148028NB-I00.


\end{document}